\documentclass[letterpaper, 10 pt, journal, twoside]{IEEEtran} 

\usepackage{amsthm,amsmath,amsfonts,mathtools,enumerate,hyperref,dsfont,tikz,subfig,hhline,array}
\usepackage[american]{circuitikz}
\usepackage{booktabs, makecell}

\ctikzset{bipoles/length=0.7cm, bipoles/thickness=1}
\newcommand\esymbol[1]{\begin{circuitikz}
\draw (0,0) to [#1] (0.6,0); \end{circuitikz}}

\newcommand{\diag}{\mathop{\bf diag}}

\newcommand{\argmin}{\mathop{\rm argmin}}
\newcommand{\argmax}{\mathop{\rm argmax}}

\newcommand{\norm}[1]{\left\lVert#1\right\rVert}
\newcommand{\mnorm}[1]{{\left\vert\kern-0.25ex\left\vert\kern-0.25ex\left\vert #1 
    \right\vert\kern-0.25ex\right\vert\kern-0.25ex\right\vert}}

\newtheorem{theorem}{Theorem}
\newtheorem{lemma}{Lemma}

\newtheorem{remark}{Remark}

\newtheorem{assumption}{Assumption}

\newcommand{\eg}{{\it e.g.}}
\newcommand{\ie}{{\it i.e.}}

\ifCLASSINFOpdf
\else
\fi
\begin{document}
%
\title{RLC Circuits based Distributed Mirror Descent Method}
%
%
%

\author{Yue~Yu and Beh\c{c}et~A\c{c}\i kme\c{s}e
\thanks{Accepeted to IEEE Control Systems Letters, available at \url{https://ieeexplore.ieee.org/document/8993740}}
}

\maketitle
\thispagestyle{empty} 
\begin{abstract}
We consider distributed optimization with smooth convex objective functions defined on an undirected connected graph. Inspired by mirror descent mehod and RLC circuits, we propose a novel distributed mirror descent method. Compared with mirror-prox method, our algorithm achieves the same \(\mathcal{O}(1/k)\) iteration complexity with only half the computation cost per iteration. We further extend our results to cases where a) gradients are corrupted by stochastic noise, and b) objective function is composed of both smooth and non-smooth terms. We demonstrate our theoretical results via numerical experiments.
\end{abstract}

\begin{IEEEkeywords}
Distributed optimization, mirror descent
\end{IEEEkeywords}

\section{Introduction}
\IEEEPARstart{D}{istributed} optimization over a graph aims to optimize sum of nodal convex objective functions via local computation on nodes and efficient communication on edges, which arises in a variety of engineering applications \cite{li2002detection,accikmecse2014decentralized,lesser2012distributed,gholami2016decentralized,yahya2017collective}. Recently there has been an increasing interest in solving distributed optimization problem using distributed mirror descent method 
\cite{duchi2012dual,xi2014distributed,li2016distributed,wang2018distributed,li2018stochastic,doan2019convergence,yuan2018optimal,nemirovski2004prox,juditsky2011solving,he2015mirror,yu2018bregman,yu2018mass,yu2019stochastic}. These methods not only include distributed (projected) subgradient method \cite{nedic2009distributed,nedic2010constrained} as special cases, but can also achieve faster convergence \cite{gholami2016decentralized,yahya2017collective}. 

Although they all revolve around mirror descent method \cite{nemirovsky1983problem,beck2003mirror}, different distributed mirror descent methods use different assumptions on objective functions, consensus dynamics, and discretization schemes. See Tab.~\ref{tab: comparison} for a detailed comparison. For example, methods in \cite{duchi2012dual,xi2014distributed,li2016distributed,wang2018distributed,li2018stochastic} add a heat equation dynamics to mirror descent step. Although using an efficient Euler-forward scheme and suitable for non-smooth objective functions, these methods converge slowly with iteration complexity \(\mathcal{O}(1/\sqrt{k})\), which can be improved to \(\mathcal{O}(1/k)\) by additionally assuming strong convexity \cite{yuan2018optimal}. On the other hand, the mirror-prox method \cite{nemirovski2004prox,juditsky2011solving,he2015mirror} solves distributed optimization as a saddle point problem, leading to a wave equation dynamics, and achieves a \(\mathcal{O}(1/k)\) iteration complexity for smooth objective functions. However, mirror-prox method uses a predictor-corrector scheme, whose per-iteration computation cost is twice as much as Euler-forward scheme. Recently, inspired by distributed alternating directional method of multipliers \cite{boyd2011distributed,wei2012distributed,meng2015proximal}, there are few attempts of adding damping to the wave equation dynamics used by mirror-prox method \cite{yu2018bregman,yu2018mass,yu2019stochastic}. The resulting methods can achieve \(\mathcal{O}(1/k)\) iteration complexity for non-smooth objective functions, but strictly rely on computationally prohibitive Euler-backward scheme that requires optimizing penalized objective function at each iteration. 

Based on the above comparison, this work is motivated by the following question: 

\emph{For smooth distributed optimization, is it possible to upgrade the predictor-corrector scheme used by mirror-prox method to the more efficient Euler-forward scheme?}

In this letter, we answer this question affirmatively and make the following contributions.
\begin{enumerate}
    \item We propose a novel distributed mirror descent method for smooth convex objective function using Euler-forward discretization of RLC circuits dynamics. Compared with mirror-prox method \cite{nemirovski2004prox}, our algorithm achieves the same \(\mathcal{O}(1/k)\) iteration complexity using only half the computation cost per iteration.
    \item We prove that our algorithm also converges in expectation when the gradients are corrupted by stochastic noise with zero mean and bounded variance.
    \item We further extend our results to cases where the objective function is composed of both smooth and non-smooth terms.
\end{enumerate}
We demonstrate our theoretical findings via numerical experiments. Our results extend the previous work \cite{yu2018mass} to noisy gradients and composite objective scenarios. 

The rest of the paper is organized as follows. After introducing necessary preliminaries in \S\ref{sec: preliminaries}, we discuss our algorithm and its convergence in \S\ref{sec: method} and, respectively, \S\ref{sec: convergence}, which are further extended to composite objective optimization in \S\ref{sec: composite}. We demonstrate our results via numerical experiments in \S\ref{sec: experiments} before concluding in \S\ref{sec: conclusion}.  

\begin{table*}[t]
\caption{Comparison of different distributed mirror descent methods.}
  \centering
  \begin{tabular}{|c|c|c|c|c|}
  \hline
  Algorithm & Objective function & Dynamics & Discretization & Convergence \\ 
  \hhline{=====}
  \cite{duchi2012dual,xi2014distributed,li2016distributed,wang2018distributed,li2018stochastic,doan2019convergence} & non-smooth, convex & heat eqn. & Euler-forward  &
  \(\mathcal{O}(1/\sqrt{k})\)\\
  \cite{yuan2018optimal} & non-smooth, strongly convex & heat eqn. & Euler-forward  & \(\mathcal{O}(1/k)\)\\
\cite{nemirovski2004prox,juditsky2011solving,he2015mirror}  & smooth, convex & wave eqn. & predictor-corrector  & \(\mathcal{O}(1/k)\)\\
\cite{yu2018bregman,yu2018mass,yu2019stochastic}  & non-smooth, convex & damped wave eqn. & Euler-backward  & \(\mathcal{O}(1/k)\)\\
  This work & smooth, convex & damped wave eqn. & Euler-forward  & \(\mathcal{O}(1/k)\)\\
  \hline
  \end{tabular}
  
  \label{tab: comparison}
\end{table*}

\section{Preliminaries}\label{sec: preliminaries}
Let \(\mathbb{R}\) denote the real numbers, \(\mathbb{R}^n\) the \(n\)-dimensional real numbers. Let \(\cdot^\top\) denote the matrix (and vector) transpose, \(\langle x, y\rangle= x^\top y\) and \(\norm{x}_2=\sqrt{\langle x, x\rangle}\) the inner product and, respectively, \(\ell_2\) norm. Let \(\diag(x)\in\mathbb{R}^{n\times n}\) denote the diagonal matrix with diagonal elements \(x\in\mathbb{R}^n\), \(I_n\) the \(n\times n\) identity matrix, \(\mathbf{1}_n\in\mathbb{R}^n\) the vector of all \(1\)'s, \(\otimes\) the Kronecker product, and \(\mathds{E}[\cdot]\) the expectation.

\subsection{Graph theory}
An undirected graph \(\mathcal{G}=(\mathcal{V}, \mathcal{E})\) consists of a node set \(\mathcal{V}\) and an edge set \(\mathcal{E}\), where an edge is a pair of distinct nodes in \(\mathcal{V}\). For an arbitrary orientation on \(\mathcal{G}\), \ie , each edge has a head and a tail, the \(|\mathcal{V}|\times |\mathcal{E}|\) incidence matrix is denoted by \(E(\mathcal{G})\). Columns of \(E(\mathcal{G})\) are indexed by the edges in \(\mathcal{E}\), and the entry on their \(i\)-th row is ``\(1\)" if node \(i\) is the head of the edge, ``\(-1\)" if it is its tail, and \(0\) otherwise.
If graph \(\mathcal{G}\) is connected, the nullspace of \(E(\mathcal{G})^\top\) is spanned by \(\mathbf{1}_{|\mathcal{V}|}\) \cite{mesbahi2010graph}.

\subsection{Convex Analysis}
We will use the following results from convex analysis (see, \eg, \cite{rockafellar2015convex}).
Given a convex set \(X\), define the normal cone \(N_X(x)\) at \(x\in X\) as follows
\begin{equation}\label{def: normal cone}
  N_{X}(x)=\{u|\, \langle u, x'-x\rangle\leq 0, \forall x'\in X\}.
\end{equation}
We say \(X\) is the domain of function \(f\), \ie, \(f:X\to\mathbb{R}\), if \(X=\{x|f(x)< \infty\}\). A function \(f:X\to\mathbb{R}\) is convex over its convex domain \(X\) if and only if \(f(\alpha x+(1-\alpha)x')\leq \alpha f(x)+(1-\alpha)f(x')\)
for all \(x, x'\in X\) and \(\alpha\in[0, 1]\). The subdifferential of function \(f\) at \(x\in X\) is defined as
\begin{equation}\label{def: subdifferential}
    \partial f(x)=\{u|f(x')-f(x)\geq \langle u, x'-x\rangle, \forall x'\in X\}.
\end{equation}
If function \(f\) is continuously differentiable over \(X\), then its subdifferential reduces to a singleton, \ie, \(\partial f(x)=\{\nabla f(x)\}\), and \(f\) is convex if and only if, for all \( x, x'\in X,\)
\begin{equation}
    f(x')\geq f(x)+\langle \nabla f(x), x'-x\rangle.
    \label{eqn: convexity}
\end{equation}
In addition, we say \(f\) is \(\beta\)-smooth if \(\frac{\beta}{2}\norm{\cdot}_2^2-f\) is also convex, which implies that, for all \( x, x'\in X,\)
\begin{equation}
     \textstyle f(x')\leq f(x)+\langle\nabla f(x), x'-x\rangle+\frac{\beta}{2}\norm{x'-x}_2^2, 
    \label{eqn: smoothness}
\end{equation}
The Bregman divergence associated with differentiable convex function \(\psi: X\to \mathbb{R}\) is defined as \cite{censor1981iterative}
\begin{equation}
    B_{\psi}(x', x)=\psi(x')-\psi(x)-\langle\nabla \psi(x), x'-x\rangle.\label{eqn: Bregman def}
\end{equation}
Applying \eqref{eqn: Bregman def} to three points \( x, x', x^+\in X\), we can show
\begin{equation}
     \label{eqn: three point}
     \begin{aligned}
     &B_\psi(x', x)-B_\psi(x', x^+)-B_\psi(x^+, x)\\
     =&\langle \nabla\psi(x^+)-\nabla\psi(x), x'-x^+\rangle.
     \end{aligned}
 \end{equation}

We will also use the following results.

\begin{lemma}[Thm. 27.4, \cite{rockafellar2015convex}]\label{lem: optimality} 
If \(f\) is a proper convex function, \(X\) is a closed convex set, then \(x'=\argmin_{x\in X} f(x)\) if and only if \(0\in \partial f(x')+N_X(x')\), \ie, there exists \(v\in\partial f(x')\) such that \( \langle v, x-x'\rangle\geq 0\) for all \(x\in X\).
\end{lemma}

\section{Method}
\label{sec: method}
In this section, we introduce our main algorithm inspired by RLC circuits dynamics. We first define the following two matrices that will be used frequently later.
\begin{equation}
\begin{aligned}
    E_l(\mathcal{G})=&(E(\mathcal{G})\diag(\sqrt{l}))\otimes I_n\\
    L_r(\mathcal{G})=& (E(\mathcal{G}\diag(r)E(\mathcal{G})^\top)\otimes I_n
\end{aligned}\label{eqn: edge weights}
\end{equation}
where \(l, r\in\mathbb{R}^{|\mathcal{E}|}\) are element-wise positive vectors, \(\sqrt{l}\) is the element-wise square root of \(l\). With this definition, we introduce the following distributed optimization problem,
\begin{equation}
    \begin{array}{ll}
    \underset{x}{\mbox{minimize}} & f(x)\coloneqq\sum_{i\in\mathcal{V}} f_i(x_i) \\
    & E_l(\mathcal{G})^\top x=0,\enskip x\in X\coloneqq X_0^{|\mathcal{V}|},  
    \end{array}\label{opt: dist opt}
\end{equation}
where \(x=[x_1^\top, \ldots, x_{|\mathcal{V}|}^\top]^\top\); \(X_0^{|\mathcal{V}|}\) is the Cartesian product of \(|\mathcal{V}|\) copies of closed convex set \(X_0\subseteq \mathbb{R}^n\); \(f_i:X_0\to\mathbb{R}\) is a convex differentiable function available to node \(i\) only. The idea of formulation \eqref{opt: dist opt} ensures local variable \(x_i\) agree with each other using constraint \(E_l(\mathcal{G})^\top x=0\).

\begin{figure}[h]
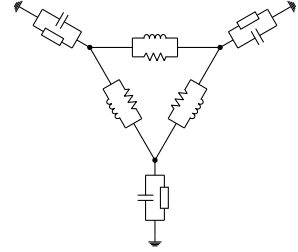

     \centering
     \ctikzset{bipoles/length=0.7cm, bipoles/thickness=1}
\begin{tikzpicture}[scale=0.5]
   
\begin{scope}[rotate=0,transform shape]
    \input{method/RLC/subRLC.tex}
\end{scope}

\begin{scope}[rotate=120,transform shape]
    \input{method/RLC/subRLC.tex}
\end{scope}

\begin{scope}[rotate=-120,transform shape]
    \input{method/RLC/subRLC.tex}
\end{scope}

\end{tikzpicture}
     \caption{An illustration of RLC circuits.}
     \label{fig: RLC}
 \end{figure}
Aiming to solve \eqref{opt: dist opt}, we consider a conceptual RLC circuits defined on graph \(\mathcal{G}=(\mathcal{V}, \mathcal{E})\) as follows.\footnote{An equivalent mechanical network model was used in \cite{yu2018mass}.} Let each node \(i\in\mathcal{V}\) denote a pin with electrical potential \(x_i(t)\) at time \(t\). Suppose a) between each pin \(i\) and ground (zero potential point), we add a nonlinear capacitor in parallel with a nonlinear resistor, mapping voltage \(x_i(t)\) to current \(\frac{d}{dt}\nabla \psi_0(x_i(t))\) for a differentiable function \(\psi_0\) and, respectively, \(\nabla f_i(x_i)\); b) on each edge \(\{ij\}\in\mathcal{E}\), we add a linear inductor in parallel with a linear resistor, mapping voltage \(x_i-x_j\) to current \(l_{\{ij\}}\int_0^t(x_i(\tau)-x_j(\tau))d\tau\) and, respectively, \(r_{\{ij\}}(x_i-x_j)\); see Fig.~\ref{fig: RLC} and Tab.~\ref{tab: RLC} for an illustration. Then Kirchoff current law says the total current flowing through pin \(i\) should sum to zero, which gives the following relation
\begin{table}[h]
\caption{Voltage-current relation of RLC units}
    \label{tab:RLC}
    \centering
    \ctikzset{bipoles/length=0.42cm}
\begin{tabular}{|c|c|c|c|}
    \hline
    \small
    type & symbol & voltage & current \\
    \hhline{====}
   resistor & \esymbol{european resistor} & \(x_i\) & \(\nabla f_i(x_i)\)\\
   capacitor & \esymbol{capacitor} & \(x_i\) & \(\frac{d}{dt}\nabla\psi_0(x_i)\)\\
  resistor & \esymbol{resistor} & \(x_i-x_j\) & \(r_{\{i j\}}(x_i-x_j)\)\\
 
  inductor & \esymbol{american inductor} & \(x_i-x_j\) & \(\int l_{\{i j\}}(x_i-x_j)dt\)\\ 
  
    \hline
    \end{tabular}
    \label{tab: RLC}
\end{table}  
\begin{equation}
    \begin{aligned}
    \textstyle 0=&\textstyle\frac{d}{dt}\nabla \psi_0(x_i(t))+\textstyle\sum_{j\in \mathcal{N}(i)}r_{\{ij\}}(x_i(t)-x_j(t))\\
    &+\nabla f_i(x_i(t))+\textstyle\sum_{j\in \mathcal{N}(i)}l_{\{ij\}}\int_0^t(x_i(\tau)-x_j(\tau))d\tau,
    \end{aligned}\label{eqn: KCL each pin}
\end{equation}
where \(j\in\mathcal{N}(i)\) if and only if \(\{ij\}\in\mathcal{E}\). Let \(x(t)=[x_1(t)^\top, \ldots, x_{|\mathcal{V}|}(t)^\top]^\top\), \(u(t)=E_l(\mathcal{G})^\top \int_0^tx(\tau)d\tau\), \(\psi(x)=\sum_{i\in\mathcal{V}}\psi_0(x_i)\) and \(f(x)=\sum_{i\in\mathcal{V}} f_i(x_i)\), we can write \eqref{eqn: KCL each pin} for all \(i\in\mathcal{V}\) compactly as follows.  
\begin{equation}\label{eqn: KCL}
\begin{aligned}
\textstyle\frac{d}{dt}\nabla\psi(x(t))=&-L_r(\mathcal{G})x(t)-E_l(\mathcal{G})u(t)-\nabla f(x(t)),\\
 \textstyle\frac{d}{dt}u(t)=&E_l(\mathcal{G})^\top x(t).
\end{aligned}
\end{equation}
If we let \(L_l(\mathcal{G})=E_l(\mathcal{G})E_l(\mathcal{G})^\top\), \(\psi=\frac{1}{2}\norm{\cdot}^2_2\) and \(f=0\), then \eqref{eqn: KCL} simplifies to the following damped wave equation, where the \emph{Laplacian operator} is approximated by discretization via Laplacian matrix \(L_r(\mathcal{G})\) and \(L_l(\mathcal{G})\).  
\begin{equation}
    \textstyle \frac{d^2}{dt^2}x(t)=-L_r(\mathcal{G})\frac{d}{dt}x(t)-L_l(\mathcal{G})x(t)
\end{equation}
Applying Euler-forward scheme to \eqref{eqn: KCL} gives
\begin{subequations}\label{eqn: EF}
\begin{align}
    \nabla \psi(x^{k+1})=&\nabla\psi(x^k)-\alpha^k w^k, \label{eqn: EF x}\\
    u^{k+1}=&u^k+\alpha^kE_l(\mathcal{G})^\top x^{k+1}\label{eqn: EF u}
    \end{align}
\end{subequations}
where \(\alpha^k>0\) is the step size and \(w^k\) is defined as
\begin{equation}
    w^k=L_r(\mathcal{G})x^k+E_l(\mathcal{G})u^k+\nabla f(x^k).\label{eqn: w gradient}
\end{equation}
Since \(\psi:X\to\mathbb{R}\) is closed, convex and proper, we know
\(x^{k+1}=\argmin_{x\in X}\,-\langle x, \nabla \psi(x^{k+1})\rangle+\psi(x)\) \cite[Thm. 23.5]{rockafellar2015convex}. Hence  update \eqref{eqn: EF x} is equivalent to the following
\begin{equation*}
    \begin{aligned}
x^{k+1}
=&\underset{x\in X}{\argmin}\,\,\alpha^k\langle w^k, x\rangle-\langle \nabla\psi(x^k), x\rangle+\psi(x)\\
=&\underset{x\in X}{\argmin}\,\,\alpha^k\langle w^k, x\rangle+B_\psi(x, x^k).
    \end{aligned}
\end{equation*}
where \(w^k\) is given by \eqref{eqn: w gradient}. Therefore, if we further assume that \(\nabla f(x^k)\) in \(w^k\) is corrupted by an additive noise vector \(\eta^k\in\mathbb{R}^{|\mathcal{V}|n}\), then \eqref{eqn: EF} gives the following algorithm,
\begin{equation}
    \begin{aligned}
        w^k=&L_{r}(\mathcal{G})x^k+E_{l}(\mathcal{G})u^k+\nabla f(x^k)+\eta^k,\\
        x^{k+1}=&\underset{x\in X}{\argmin}\,\, \alpha^k\langle w^k, x\rangle+B_\psi(x, x^k),\\
        u^{k+1}=&u^k+\alpha^k E_{l}(\mathcal{G})^\top x^{k+1}.
    \end{aligned}\label{alg: RLC}\tag{RLC}
\end{equation}
The \(x\)-update in \eqref{alg: RLC} minimizes a linearized version of the cost function while penalizing the distance from \(x^k\) where such linearization is valid, which is a popular scheme under smoothness assumption \cite{nemirovski2004prox,juditsky2011solving,he2015mirror}. Algorithm also \eqref{alg: RLC} allows distributed implementation since equation \eqref{eqn: KCL} is fully separable with respect to different nodes in \(\mathcal{G}\) and no coupling is introduced during discretization.

\begin{remark}
Alternatively, applying the predictor-corrector scheme to the undamped version of \eqref{eqn: KCL}, \ie, \(r=0\), gives the following mirror-prox method \cite{nemirovski2004prox}
\begin{equation}
    \begin{aligned}
        y^{k}=&\underset{y\in X}{\argmin}\,\, \alpha^k\langle E_{l}(\mathcal{G})u^k+\nabla f(x^k), y\rangle+B_\psi(y, x^k),\\
        v^{k}=&u^k+\alpha^k E_{l}(\mathcal{G})^\top x^k,\\
        x^{k+1}=&\underset{x\in X}{\argmin}\,\, \alpha^k\langle E_{l}(\mathcal{G})v^k+\nabla f(y^k), x\rangle+B_\psi(x, x^k),\\
        u^{k+1}=&u^k+\alpha^k E_{l}(\mathcal{G})^\top y^{k}.
    \end{aligned}\label{alg: mirror-prox}\tag{mirror-prox}
\end{equation}
Compared with \eqref{alg: mirror-prox}, iterations in \eqref{alg: RLC} compute \(\nabla f(x)\) (gradient evaluating) and \(E_l(\mathcal{G})^\top x\) (information exchange among neighboring nodes) only half as often by adding an \(L_r(\mathcal{G})x\) term at virtually no additional cost.
\end{remark}

\section{Convergence}
\label{sec: convergence}
In this section, we establish the convergence of \eqref{alg: RLC}. We first make the following technical assumptions.
\begin{assumption}
\begin{enumerate}
    \item \(\mathcal{G}=(\mathcal{V}, \mathcal{E})\) is undirected and connected. Edge weights \(l, r\in\mathbb{R}^{|\mathcal{E}|}\) used in \eqref{eqn: edge weights} are element-wise positive. Let \(\frac{1}{\gamma}=\max\{\alpha|\, \alpha l \leq r\}\) and \(\lambda\) be the largest eigenvalue of matrix \(L_r(\mathcal{G})\).
    \item \(X_0\subseteq\mathbb{R}^n\) is a closed convex set,\footnote{In general, \(X_0\) can be a closed convex subset of any Hilbert spaces. Here we assume this space is \(\mathbb{R}^n\)  for simplicity.} \(\psi_0:X_0\to\mathbb{R}\) is proper, closed, continuously differentiable and \(1\)-strongly convex, \ie, \(\psi_0-\frac{1}{2}\norm{\cdot}_2^2\) is convex over \(X_0\).
    \item  For all \(i\in\mathcal{V}\), \(f_i: X_0\to \mathbb{R}\) is continuously differentiable, convex and \(\beta\)-smooth, \ie,  both \(f_i\) and \(\frac{\beta}{2}\norm{\cdot}_2^2-f_i\) are convex over \(X_0\). There exists \(x^\star, u^\star\) such that
\begin{subequations}
    \begin{align}
    &E_l(\mathcal{G})^\top x^\star=0, \enskip x^\star\in X,\label{kkt: primal}\\
    -&E_l(\mathcal{G})u^\star-\nabla f(x^\star)\in N_{X}(x^\star).\label{kkt: dual}
    \end{align}
\end{subequations}
\end{enumerate}\label{asp: basic}
\end{assumption}
We define the Lagrangian of \eqref{opt: dist opt} as follows
\begin{equation}
    \ell(x, u)=f(x)+\langle E_l(\mathcal{G})u, x\rangle.
    \label{eqn: Lagrangian}
\end{equation}
Using Lemma~\ref{lem: optimality} one can show that \eqref{kkt: primal} and \eqref{kkt: dual} imply
\(u^\star=\argmax_u\,\ell(x, u)\) and \(x^\star=\argmin_{x\in X}\,\ell(x, u)\) for any \(x\) and \(u\). Further, \eqref{kkt: primal} also implies that \(\ell(x^\star, u^\star)=\ell(x^\star, u)\) for any \(u\), hence,  
\begin{equation*}
    \begin{aligned}
    0\leq \ell(x, u^\star)-\ell(x^\star, u^\star)=\underset{u}{\mbox{max}}\,\ell(x, u)-\underset{x\in X}{\mbox{min}}\,\ell(x, u).
    \end{aligned}
\end{equation*} 
for all \(x\in X\) and \(u\in\mathbb{R}^{|\mathcal{V}|n}\).
Non-negative function \(\ell(x, u^\star)-\ell(x^\star, u^\star)\) is also know as the \emph{running duality gap function}, which measures the quality of solution \(x\) \cite{nemirovski2004prox,meng2015proximal}. 

We will use the following energy function, which represents the sum of electrical and magnetic energy of the RLC circuits used in \eqref{eqn: KCL}.
\begin{equation}
    \label{eqn: Lyapunov}
     \textstyle V(x, u)=B_\psi(x^\star, x)+\frac{1}{2}\norm{u-u^\star}_2^2.
\end{equation}

The following lemma shows that, along trajectories generated by \eqref{alg: RLC}, the variation of \(V(x, u)\) is controlled by \(\ell(x, u^\star)-\ell(x^\star, u^\star)\) up to a discretization error. 

\begin{lemma}
\label{lem: dissipation}
Suppose Assumption~\ref{asp: basic} holds. If \(0<\alpha^k\leq \frac{1}{\gamma}\), then along the trajectories generated by \eqref{alg: RLC}, we have
\begin{equation}\label{eqn: dissipation}
    \begin{aligned}
    & V(x^{k+1}, u^{k+1})-V(x^k, u^k)
    \\
    \leq
    & -\alpha^k(\ell(x^{k+1}, u^\star)-\ell(x^\star, u^\star))\\
    & \textstyle-\alpha^k\langle \eta^k, x^{k+1}-x^\star\rangle-\frac{1-\alpha^k(\lambda+\beta)}{2}\norm{x^{k+1}-x^k}_2^2,
    \end{aligned}
\end{equation}
where \(\ell(x, u)\) and \(V(x, u)\) are given by \eqref{eqn: Lagrangian} and \eqref{eqn: Lyapunov}.
\end{lemma}

\begin{proof}
The structure of the proof is as follows. We first bound the variation of \(V(x, u)\) between two consecutive iterations by the sum of two inequalities, which can be further simplified under Assumption~\ref{asp: basic}.  

First, using Lemma~\ref{lem: optimality} we can show that the \(x\)-update in \eqref{alg: RLC} implies the following
\begin{equation*}
0\leq\langle \alpha^kw^k+\nabla\psi(x^{k+1})-\nabla\psi(x^k), x^\star-x^{k+1}\rangle.
\end{equation*}
Substituting the \(w\)-update in \eqref{alg: RLC} and \eqref{eqn: three point} (with \(x=x^k, x'=x^\star, x^+=x^{k+1}\)) into the above inequality gives
\begin{equation}\label{eqn: B-1}
    \begin{aligned}
    &B_\psi(x^\star, x^{k+1})-B_\psi(x^\star, x^k)+B_\psi(x^{k+1}, x^k)\\
    \leq 
    &-\alpha^k\langle L_r(\mathcal{G})x^k+E_l(\mathcal{G})u^k+\nabla f(x^k)+\eta^k, x^{k+1}-x^\star\rangle,
    \end{aligned}
\end{equation}
where we also use \eqref{kkt: primal}. 
Second, let \(L_l(\mathcal{G})=E_l(\mathcal{G})E_l(\mathcal{G})^\top\), then using \eqref{eqn: three point} with \(\psi=\frac{1}{2}\norm{\cdot}_2^2\) we can show
\begin{equation}\label{eqn: B-2}
\begin{aligned}
    &\textstyle\frac{1}{2}\norm{u^{k+1}-u^\star}_2^2-\frac{1}{2}\norm{u^k-u^\star}_2^2\\
    =& \textstyle\langle u^k-u^\star, u^{k+1}-u^k\rangle+\frac{1}{2}\norm{u^{k+1}-u^k}_2^2\\
    = &\textstyle\alpha^k\langle E_{l}(\mathcal{G})(u^{k}-u^\star), x^{k+1}-x^\star \rangle\\
    &\textstyle+\frac{\alpha^k}{2}\langle \alpha^k L_{l}(\mathcal{G}) x^{k+1}, x^{k+1}\rangle,
    \end{aligned}
\end{equation}
where the last step uses the \(u\)-update in \eqref{alg: RLC} and \eqref{kkt: primal}.
Further, we can show the following three inequalities, each using one part of Assumption~\ref{asp: basic}.

\emph{Part 1):} Since \(0<\alpha^k\leq \frac{1}{\gamma}=\max\{\alpha|\, \alpha l \leq r\}\), vector \(r-\alpha^k l\) is element-wise non-negative and matrix \(L_r(\mathcal{G})-\alpha^kL_l(\mathcal{G})=(E(\mathcal{G})\diag(r-\alpha^kl)E(\mathcal{G})^\top)\otimes I_n\) is positive semi-definite, which implies
\begin{equation}\label{eqn: B-3}\textstyle0\leq \frac{\alpha^k}{2}\langle (L_r(\mathcal{G})-\alpha^kL_l(\mathcal{G}))x^{k+1}, x^{k+1}\rangle.
\end{equation}

\emph{Part 2):} Since \(\psi_0\) is \(1\)-strongly convex, we can apply \eqref{eqn: convexity} to convex function \(\psi-\frac{1}{2}\norm{\cdot}_2^2\), which gives
\begin{equation}\label{eqn: B-4}
    \textstyle\frac{1}{2}\norm{x^{k+1}-x^k}_2^2\leq B_\psi(x^{k+1}, x^k) .
\end{equation}

\emph{Part 3):}
Since \(f_i\) is convex and \(\beta\)-smooth for all \(i\in\mathcal{V}\) and \(\lambda\) is the largest eigenvalue of \(L_r(\mathcal{G})\), function \(h(x)\coloneqq f(x)+\frac{1}{2}\langle L_r(\mathcal{G})x, x\rangle\) is convex and \((\beta+\lambda)\)-smooth. Hence we can apply inequality \eqref{eqn: convexity} and, respectively, inequality \eqref{eqn: smoothness} to function \(h(x)\), which gives the following.
\begin{equation*}
    \begin{aligned}
    \alpha\langle \nabla h(x), x'-x\rangle\leq& \alpha(h(x')-h(x)),\\ 
    \alpha\langle\nabla h(x), x-x^+\rangle\leq& \alpha \big(h(x)-h(x^+)+\textstyle\frac{\beta+\lambda}{2}\norm{x^+-x}_2^2\big),
    \end{aligned}
\end{equation*}
for all \(x, x^+, x'\in X\) and \(\alpha>0\). Summing up the above two inequalities with \(x=x^k\), \(x'=x^\star\), \(x^+=x^{k+1}\), \(\alpha=\alpha^k\), and recalling that \(h(x)= f(x)+\frac{1}{2}\langle L_r(\mathcal{G})x, x\rangle\), we obtain 
\begin{equation}\label{eqn: B-5}
    \begin{aligned}
    &-\alpha^k\langle \nabla f(x^k)+L_r(\mathcal{G})x^k, x^{k+1}-x^\star \rangle\\
    \leq &\textstyle-\alpha^k(f(x^{k+1})-f(x^\star))-\frac{\alpha^k}{2}\langle L_r(\mathcal{G})x^{k+1}, x^{k+1}\rangle\\
    &\textstyle+\frac{\alpha^k(\beta+\lambda)}{2}\norm{x^{k+1}-x^{k}}_2^2,
    \end{aligned}
\end{equation}
where we also use the fact that \(h(x^\star)=f(x^\star)\).

Finally, summing up \eqref{eqn: B-1}--\eqref{eqn: B-5} gives \eqref{eqn: dissipation}.
\end{proof}
Based on Lemma~\ref{lem: dissipation}, we will prove that \(\ell(x, u^\star)\) converges to \(\ell(x^\star, u^\star)\) in expectation along trajectories generated by \eqref{alg: RLC}. The key idea is upper bounding the last two terms on the right hand side of \eqref{eqn: dissipation}.

We first start with the noiseless case where \(\eta^k=0\).
\begin{theorem}[Noiseless gradient]
Suppose Assumption~\ref{asp: basic} hold. Along the trajectories generated by \eqref{alg: RLC}, if \(0<\alpha^k\equiv \alpha\leq \min\{\frac{1}{\beta+\lambda}, \frac{1}{\gamma}\} \) and \(\eta^k=0\) for all \(k\), then
\begin{equation*}
    \textstyle\ell(\overline{x}^K, u^\star)-\ell(x^\star, u^\star)\leq \frac{V(x^1, u^1)}{\alpha K},
\end{equation*}
where \(\overline{x}^K=\frac{1}{K}\sum_{k=1}^K x^{k+1}\), \(\ell(x, u)\) and \(V(x, u)\) are given by \eqref{eqn: Lagrangian} and, respectively, \eqref{eqn: Lyapunov}.\label{thm: gradient}
\end{theorem}

\begin{proof}
If \(\eta^k=0\) and \(\alpha^k\equiv \alpha\leq\frac{1}{\beta+\lambda}\), then \eqref{eqn: dissipation} reduces to
\begin{equation*}
    \alpha(\ell(x^{k+1}, u^\star)-\ell(x^\star, u^\star))\leq V(x^k, u^k)-V(x^{k+1}, u^{k+1}).
\end{equation*}
Summing up the above inequality from \(k=1\) to \(k=K\) we have \(\alpha\sum_{k=1}^K(\ell(x^{k+1}, u^\star)-\ell(x^\star, u^\star))\leq V(x^1, u^1)\), which, combined with Jensen's inequality  \(\ell(\overline{x}^K, u^\star)\leq\frac{1}{K}\sum_{k=1}^K\ell(x^{k+1}, u^\star)\), completes the proof.
\end{proof}

Theorem~\ref{thm: gradient} shows that the step sizes for  algorithm \eqref{alg: RLC} (noiseless case) are bounded by \(\min\{\frac{1}{\beta+\lambda}, \frac{1}{\gamma}\}\). In comparison, one can show that, under Assumption~\ref{asp: basic}, the step sizes for algorithm \eqref{alg: mirror-prox} are bounded by \(\min\{\frac{1}{2\beta}, \frac{1}{2\sqrt{\gamma\lambda}}\}\)\cite[Thm. 5.2]{bubeck2015convex}.

The following theorem extends Theorem~\ref{thm: gradient} to cases where \(\eta^k\) has zero mean and bounded variance.

\begin{theorem}[Noisy gradient]\label{thm: stochastic gradient}
Suppose Assumption~\ref{asp: basic} hold. Along the trajectories generated by \eqref{alg: RLC}, if \(0<\alpha^k\leq \min\{\frac{1}{2(\beta+\lambda)}, \frac{1}{\gamma}\}\), \(\eta^k\) is an independent random variable with \(\mathds{E}[\eta^k]= 0\), \(\mathds{E}\big[\norm{\eta^k}_2^2\big]\leq \sigma^2\) for all \(k\), then
\begin{equation*}
    \textstyle\mathds{E}[\ell(\overline{x}^K, u^\star)]-\ell(x^\star, u^\star)\leq \frac{V(x^1, u^1)+\sigma^2\sum_{k=1}^K (\alpha^k)^2}{\sum_{k=1}^K\alpha^k},
\end{equation*}
where \(\overline{x}^K=(\sum_{k=1}^K\alpha^k)^{-1}\sum_{k=1}^K \alpha^k x^{k+1}\), \(\ell(x, u)\) and \(V(x, u)\) are given by \eqref{eqn: Lagrangian} and, respectively, \eqref{eqn: Lyapunov}.
\end{theorem}

\begin{proof}
Observe that
\begin{equation}\label{eqn: thm2 eqn1}
    \begin{aligned}
    &-\alpha^k\langle \eta^k, x^{k+1}-x^\star\rangle\\
    =&-\alpha^k\langle \eta^k, x^{k+1}-x^k\rangle-\alpha^k\langle \eta^k, x^k-x^\star\rangle\\
    \leq &\textstyle\frac{\varepsilon\alpha^k}{2}\norm{\eta^k}_2^2+\frac{\alpha^k}{2\varepsilon}\norm{x^{k+1}-x^k}_2^2-\alpha^k\langle \eta^k, x^k-x^\star\rangle,
    \end{aligned}
\end{equation}
where the last step is by completing the square, and we let \(\varepsilon=\frac{\alpha^k}{1-\alpha^k(\beta+\lambda)}>0\). Since \(\frac{\varepsilon\alpha^k}{2}\leq (\alpha^k)^2\) when \(\alpha^k\leq \frac{1}{2(\beta+\lambda)}\), summing up \eqref{eqn: dissipation} and \eqref{eqn: thm2 eqn1} gives
\begin{equation}\label{eqn: thm2 eqn2}
    \begin{aligned}
    &V(x^{k+1}, u^{k+1})-V(x^{k}, u^{k})-(\alpha^k)^2\norm{\eta^k}_2^2\\
    \leq & -\alpha^k(\ell(x^{k+1}, u^\star)-\ell(x^\star, u^\star))-\alpha^k\langle \eta^k, x^k-x^\star\rangle).
    \end{aligned}
\end{equation}
From \eqref{alg: RLC} we know that, as random variables, \(x^k\) and \(\eta^k\) are independent of each other, hence
\[\mathds{E}[\langle \eta^k, x^k-x^\star\rangle]=\langle \mathds{E}[\eta^k], \mathds{E}[x^k-x^\star]\rangle=0.\]
Therefore taking expectation on both sides of \eqref{eqn: thm2 eqn2} gives 
\begin{equation*}
    \begin{aligned}
    &\mathds{E}[\alpha^k(\ell(x^{k+1}, u^\star)]-\alpha^k\ell(x^\star, u^\star)\\
    \leq &\mathds{E}[V(x^{k}, u^{k})]-\mathds{E}[V(x^{k+1}, u^{k+1})]+(\alpha^k)^2\sigma^2,
    \end{aligned}
\end{equation*}
where we also use the assumption that \(\mathds{E}\big[\norm{\eta^k}_2^2\big]\leq \sigma^2\). Summing up the above inequality from \(k=1\) to \(k=K\), we have \(\sum_{k=1}^K\mathds{E}[\alpha^k(\ell(x^{k+1}, u^\star)]-\alpha^k\ell(x^\star, u^\star)\leq V(x^1, u^1)+\sum_{k=1}^K(\alpha^k)^2\sigma^2\), which, combined with Jensen's inequality \(\ell(\overline{x}^K, u^\star)\leq\frac{1}{\sum_{k=1}^K\alpha^k}\sum_{k=1}^K\alpha^k\ell(x^{k+1}, u^\star)\), completes the proof.
\end{proof}

\begin{remark}\label{rem: noise floor}
Theorem~\ref{thm: stochastic gradient} shows that if \(\alpha^k\equiv \alpha\), algorithm~\ref{alg: RLC} converges to a ``noise floor" given by \(\alpha\sigma^2\), \ie, \(\lim_{K\to\infty}\mathds{E}[\ell(\overline{x}^K, u^\star)]-\ell(x^\star, u^\star)=\alpha\sigma^2\). If the maximum iteration number \(K\) is fixed in advance as a computation budget and we choose \(\alpha^k\equiv 1/\sqrt{K}\), then such noise floor also decreases with \(K\) at the rate of \(\mathcal{O}(1/\sqrt{K})\).
\end{remark}

\section{Distributed Composite Objective Optimization}
\label{sec: composite}
An important paradigm in machine learning is composite objective optimization where the objective function include, in addition to a smooth cost function, a potentially non-smooth regularization function that typically promotes desired solution properties (\eg , sparsity). Examples of composite optimization include ridge regression, lasso, and logistic regression \cite{duchi2010composite}. Combining such paradigm with problem \eqref{opt: dist opt} gives the following distributed composite objective optimization problem
\begin{equation}\label{opt: co-opt}
    \begin{array}{ll}
    \underset{x}{\mbox{minimize}} & f(x)+g(x)\coloneqq\sum_{i\in\mathcal{V}} f_i(x_i)+ \sum_{i\in\mathcal{V}} g_i(x_i)\\
    & E_l(\mathcal{G})^\top x=0,\enskip x\in X\coloneqq X_0^{|\mathcal{V}|}.  
    \end{array}
\end{equation}
where, in addition to the smooth cost function \(f_i\), each node \(i\in\mathcal{V}\) also has a possibly non-smooth convex regularization function \(g_i:X_0\to \mathbb{R}\). 

Motivated by composite objective mirror descent method \cite{duchi2010composite}, we propose the following variation of \eqref{alg: RLC}, which adds an additional regularization to the \(x\)-update, 
\begin{equation}
    \begin{aligned}
        w^k=&L_{r}(\mathcal{G})x^k+E_{l}(\mathcal{G})u^k+\nabla f(x^k)+\eta^k,\\
        x^{k+1}=&\underset{x\in X}{\argmin}\,\, \alpha^k\langle w^k, x\rangle+\alpha^kg(x)+B_\psi(x, x^k),\\
        u^{k+1}=&u^k+\alpha^k E_{l}(\mathcal{G})^\top x^{k+1}.
    \end{aligned}\label{alg: co-RLC}\tag{co-RLC}
\end{equation}

Algorithm~\eqref{alg: co-RLC} can be viewed as a distributed extension to composite objective mirror descent method \cite{duchi2010composite}; or, alternatively, a direct combination of \eqref{alg: RLC} and the algorithm proposed in \cite{yu2018mass}.

\begin{remark}\label{rem: shrinkage}
An important special case of \eqref{alg: co-RLC} is when \(X_0=\mathbb{R}^n\) and \(g_i=\eta\norm{\cdot}_1\) with \(\theta>0\) for all \(i\in\mathcal{V}\), which aims to induce sparse solution using \(\ell_1\) norm regularization. In this case, one can show that the \(x\)-update in \eqref{alg: co-RLC} reduces to \(x^{k+1}=S_{\alpha^k\theta}(x^k-\alpha^kw^k)\), where \(S_\alpha(x)\) is the shrinkage operator defined element-wise as follows \cite{duchi2010composite}
\[[S_{\alpha}(x)]_j=\text{sign}([x]_j)\max\{0, |[x]_j|-\alpha\},\]
where \([x]_j\) denote the \(j\)-th element of vector \(x\).
\end{remark}

The following theorem shows that the convergence proof of \eqref{alg: co-RLC} follows the same ideas used by Theorem~\ref{thm: gradient} and Theorem~\ref{thm: stochastic gradient}.
\begin{theorem}[Composite objective]\label{thm: composite}
Suppose \(g_i:X_0\to\mathbb{R}^n\) is closed, convex, proper for all \(i\in\mathcal{V}\). Suppose Assumption~\ref{asp: basic} hold with \eqref{kkt: dual} replaced by 
\[-E_l(\mathcal{G})u^\star-\nabla f(x^\star)\in N_X(x^\star)+\partial g(x^\star).\]
Then Theorem~\ref{thm: gradient} and Theorem~\ref{thm: stochastic gradient} hold if \eqref{alg: RLC} is replaced by \eqref{alg: co-RLC} and \(\ell(x, u^\star)\) is replaced by \(\ell(x, u^\star)+g(x)\).
\end{theorem}
\begin{proof} 
Using Lemma~\ref{lem: optimality} we can show that the \(x\)-update in \eqref{alg: RLC} implies: there exists \(v\in\partial g(x^{k+1})\) such that
\begin{equation*}
0\leq\langle \alpha^k(w^k+v)+\nabla\psi(x^{k+1})-\nabla\psi(x^k), x^\star-x^{k+1}\rangle,
\end{equation*}
Since \(v\in\partial g(x^{k+1})\), we can use \eqref{def: subdifferential} to show
\begin{equation*}
    \alpha^k\langle v, x^\star-x^{k+1}\rangle\leq \alpha^k(g(x^\star)-g(x^{k+1})).
\end{equation*}
Summing up the above two inequalities, then use
\(w\)-update in \eqref{alg: co-RLC} and \eqref{eqn: three point} we can show the following
\begin{equation}\label{eqn: B-1.1}
    \begin{aligned}
    &B_\psi(x^\star, x^{k+1})-B_\psi(x^\star, x^k)-\alpha^k(g(x^\star)-g(x^{k+1}))\\
    \leq &-B_\psi(x^{k+1}, x^k)-\alpha^k\langle L_r(\mathcal{G})x^k, x^{k+1}\rangle\\
    &-\alpha^k\langle E_l(\mathcal{G})u^k+\nabla f(x^k)+\eta^k, x^{k+1}-x^\star\rangle,
    \end{aligned}
\end{equation}
Notice that \eqref{alg: co-RLC} differs from \eqref{alg: RLC} only in its \(x\)-update, which implies \eqref{eqn: B-2}--\eqref{eqn: B-5} in Lemma~\ref{lem: dissipation} still hold as the assumptions they used are unchanged. Therefore we can sum up \eqref{eqn: B-1.1} and \eqref{eqn: B-2}--\eqref{eqn: B-5}, arriving at
\begin{equation}\label{eqn: thm3 eqn2}
    \begin{aligned}
    & V(x^{k+1}, u^{k+1})-V(x^k, u^k)\\
    \leq & \textstyle-\alpha^k\langle \eta^k, x^{k+1}-x^\star\rangle-\frac{1-\alpha^k(\lambda+\beta)}{2}\norm{x^{k+1}-x^k}_2^2\\
    &\alpha^k(\ell(x^{k+1}, u^\star)+g(x^{k+1})-\ell(x^\star, u^\star)-g(x^\star)).
    \end{aligned}
\end{equation}
Finally, applying the same arguments as in the proof of Theorem~\ref{thm: gradient} and Theorem~\ref{thm: stochastic gradient} to \eqref{eqn: thm3 eqn2} instead of \eqref{eqn: dissipation} completes the proof.
\end{proof}
\section{Numerical Experiments}
\label{sec: experiments}
In this section, we compare our algorithms against the distributed mirror descent method in \cite{li2016distributed,wang2018distributed,doan2019convergence} and mirror-prox method (and its composite objective extensions) \cite{nemirovski2004prox,juditsky2011solving,he2015mirror} via numerical examples. We first generate a random graph \(\mathcal{G}\) with \(|\mathcal{V}|=30\) and each pair of nodes are connected with probability \(0.3\). We let \(f_i(x_i)=\frac{1}{2}\norm{A_ix_i-b_i}_2^2\) for all \(i\in\mathcal{V}\), where entries of \(A_i, b_i\) are randomly generated. Using such choice on \(\mathcal{G}\) and \(f\), we construct the following two examples (where \(n=30\)). 
\begin{itemize}
    \item least squares over simplex: \eqref{opt: dist opt} with \(X_0=\{v\in\mathbb{R}^n| v\geq 0, \langle v, \mathbf{1}_n\rangle=1\}\), \(\psi_0(v)=\langle v, \ln v\rangle\).\footnote{\(\ln v\) denotes the element-wise natural logarithm of \(v\).}
    \item least squares with \(\ell_1\) regularization: \eqref{opt: co-opt} with \(g_i(v)=\frac{1}{100}\norm{v}_1\) for all \(i\in\mathcal{V}\), \(X_0=\mathbb{R}^n\), \(\psi_0(v)=\frac{1}{2}\norm{v}_2^2\).
\end{itemize}
For \eqref{alg: RLC} and \eqref{alg: co-RLC}), we set \(r=\frac{1}{10}\mathbf{1}_{|\mathcal{E}|}\), \(l=(\beta+\lambda)r\) and \(\alpha^k\equiv\frac{1}{\beta+\lambda}\); for distributed mirror descent, we use step size \(\frac{1}{\sqrt{k}}\) and doubly stochastic matrix \(P=I-\frac{1}{1+\Delta}L(\mathcal{G})\), where \(L(\mathcal{G})=E(\mathcal{G})E(\mathcal{G})^\top\) and \(\Delta\) is the largest diagonal element of \(L(\mathcal{G})\); for mirror-prox method, we use step size \(\min\{\frac{1}{2\beta}, \frac{1}{2\sqrt{\gamma\lambda}}\}\) for problem \eqref{opt: dist opt} and \eqref{opt: co-opt}, as we discussed after Theorem~\ref{thm: gradient}. In the noisy gradient case, we sample \(\eta^k\) from Gaussian distribution \(\mathcal{N}(0, \sigma I_{|\mathcal{V}|n})\) for all \(k\).

The convergence of algorithms are shown in Fig.~\ref{fig: entropy} and Fig.~\ref{fig: lasso}. We can see that the convergence of \eqref{alg: RLC} and \eqref{alg: co-RLC} behaves no worse than  mirror-prox method \cite{nemirovski2004prox} and its composite objective extension \cite{he2015mirror} using only half of the computation cost per iteration; they all reach a ``noise floor" when gradient are corrupted by noise. In the \(\ell_1\) regularized least squares case, both mirror-prox and RLC significantly outperform distributed mirror descent method since they use the shrinkage operator, as we discussed in Remark~\ref{rem: shrinkage}, rather than subgradients of \(\ell_1\) norm. 

\begin{figure}[t]
\centering
     \subfloat[\(\sigma=0\)\label{subfig: entropy1}]{%
       \includegraphics[width=0.235\textwidth]{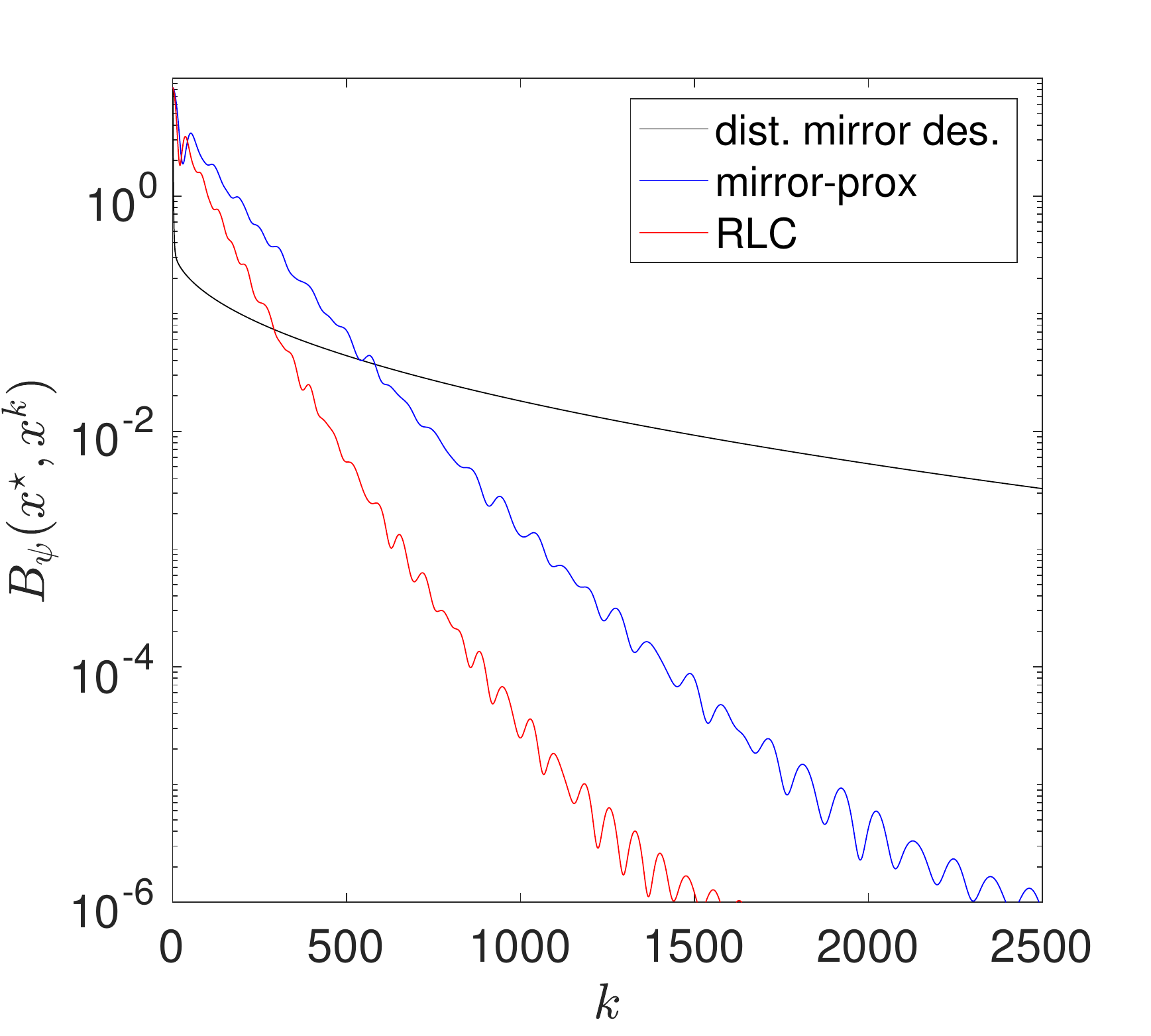}
     }
     \subfloat[\(\sigma=10^{-3}\)\label{subfig: entropy2}]{%
       \includegraphics[width=0.235\textwidth]{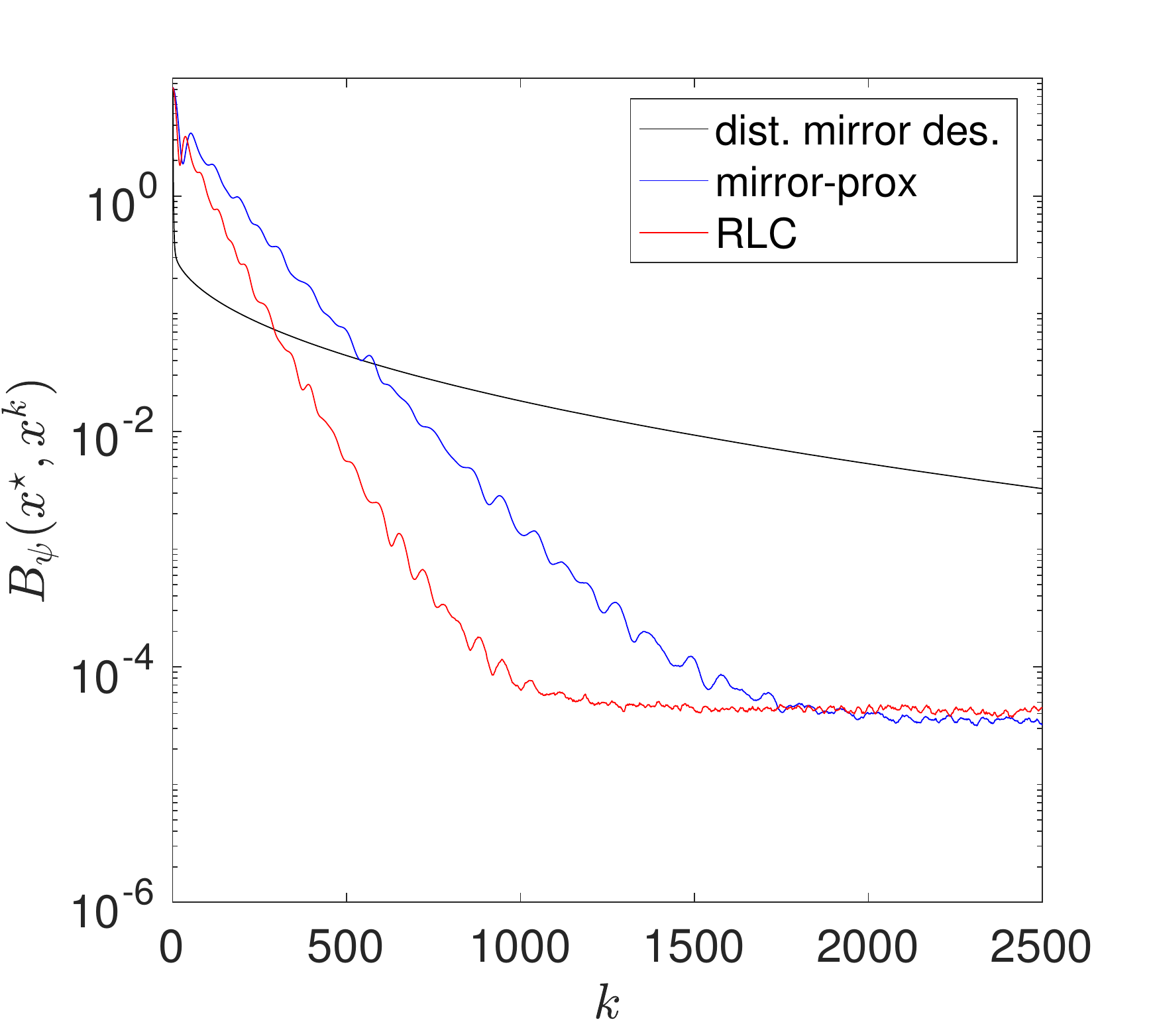}
     }
     \caption{Least squares over simplex.}
     \label{fig: entropy}
\end{figure}

\begin{figure}[t]
\centering
     \subfloat[\(\sigma=0\)\label{subfig: lasso1}]{%
       \includegraphics[width=0.235\textwidth]{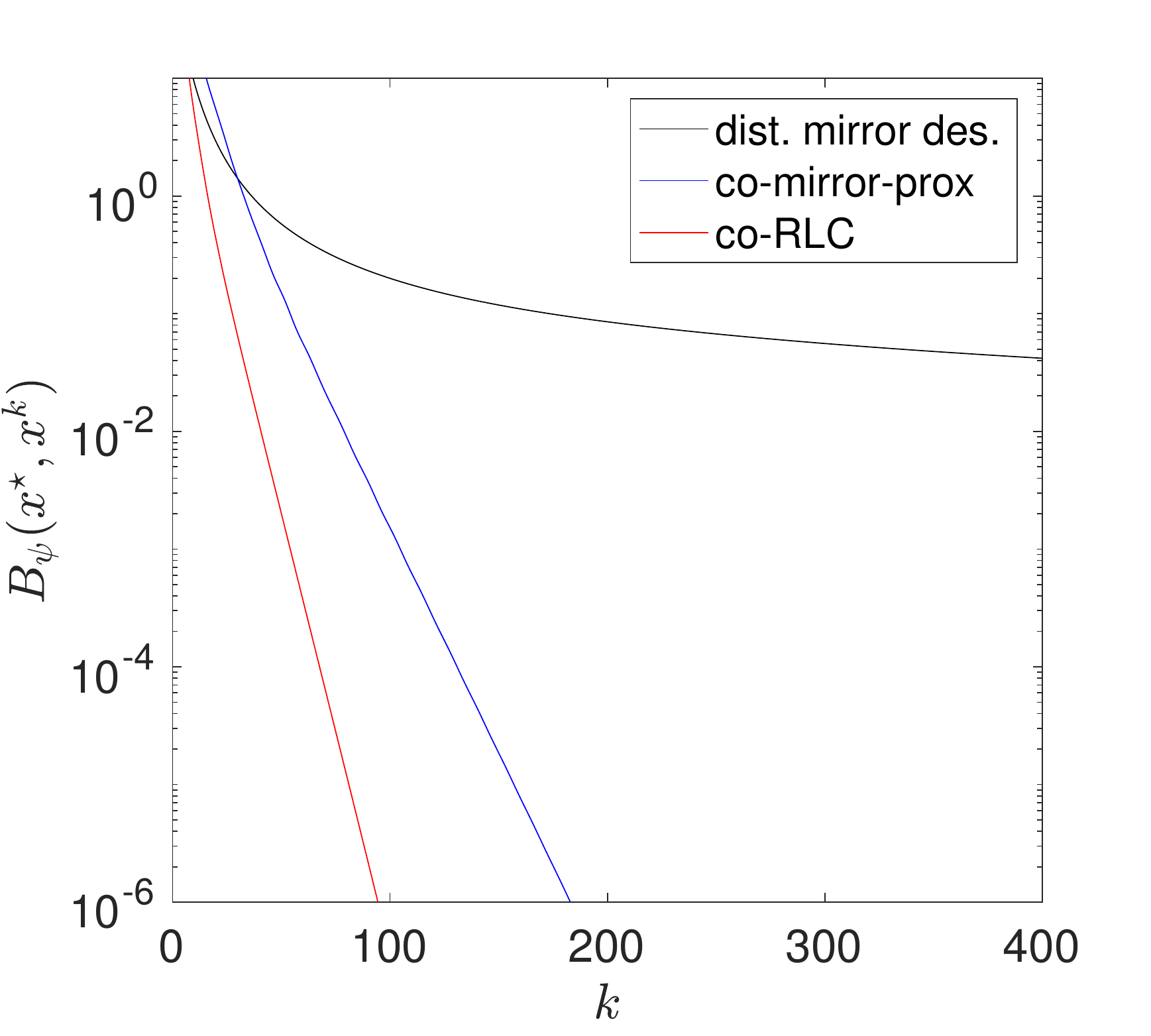}
     }
     \subfloat[\(\sigma=10^{-3}\)\label{subfig: lasso2}]{%
       \includegraphics[width=0.235\textwidth]{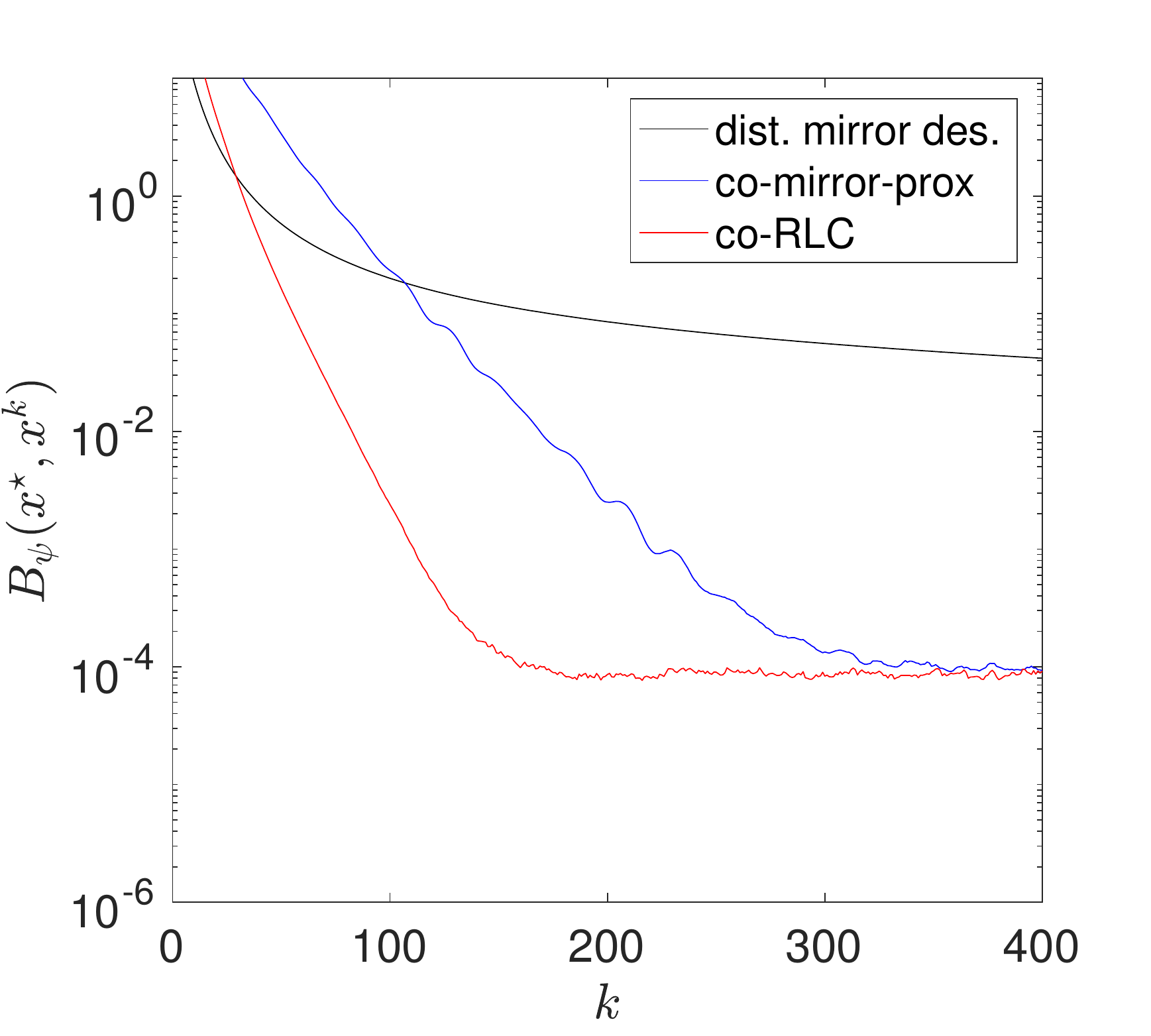}
     }
     \caption{\(\ell_1\) regularized least squares.}
     \label{fig: lasso}
\end{figure}
\section{Conclusion}
\label{sec: conclusion}

Inspired by RLC circuits, we propose a novel distributed mirror descent for smooth distributed optimization, which, compared with mirror-prox method, uses only half the per iteration computation cost. We extend our results to noisy gradients and composite objective setting. However, our results are limited to undirected graphs and it is still an open question of how to extend them to directed graphs. Our future directions include answering such question, as well as developing extensions to online distributed setting and improving convergence results under stronger assumptions.

%

\ifCLASSOPTIONcaptionsoff
  \newpage
\fi



%
\bibliographystyle{IEEEtran}
\bibliography{IEEEabrv,reference}

\end{document}